\documentclass[authoryear,preprint,review,12pt]{elsarticle}

\usepackage{epsfig}
\usepackage{graphicx}
\usepackage{amsmath}

\usepackage{amssymb}
\usepackage{amsthm}

\newtheorem{thm}{Theorem}
\newtheorem{lem}[thm]{Lemma}
\newtheorem{cor}[thm]{Corollary}
\newtheorem{rmk}[thm]{Remark}
\newproof{pf}{Proof}
\newproof{pot}{Proof of Theorem \ref{thm2}}

\journal{Computational and Applied Mathematics}

\begin{document}

\begin{frontmatter}



\title{Asymptotic Expansions for High-Contrast Linear Elasticity}


\author[unal]{Leonardo A. Poveda\corref{cor2}}
\ead{leapovedacu@unal.edu.co}

\author[unal]{Sebastian Huepo}
\ead{shuepobe@unal.edu.co}

\author[kaust]{Victor M. Calo}
\ead{victor.calo@kaust.edu.sa}

\author[unal]{Juan Galvis}
\ead{jcgalvisa@unal.edu.co}

\address[unal]{Departamento de Matem\'aticas, Universidad Nacional de Colombia, Bogot\'a, Colombia}

\address[kaust]{Center for Numerical Porous Media, Applied Mathematics \& Computational Science and Earth Sciences \& Engineering, King Abdullah University of Science and Technology\\ Thuwal 23955-6900, Kingdom of Saudi Arabia}

\cortext[cor2]{Corresponding author at: Departamento de Matem\'aticas, Universidad Nacional de Colombia, Bogot\'a, DC}


\begin{abstract}
We study linear elasticity problems with high contrast in the coefficients using asymptotic limits recently introduced. We derive an asymptotic expansion to solve heterogeneous elasticity problems in terms of the contrast in the coefficients. We study the convergence of the expansion in the $H^1$ norm.
\end{abstract}

\begin{keyword} Linear elasticity problem \sep high-contrast coefficients  \sep asymptotic expansions \sep highly inelastic inclusion \sep convergence  


\end{keyword}

\end{frontmatter}


\section{Introduction}
\label{Sec1:Introduction}

There is a growing interest in the computation of solutions of problems governed by partial differential equations with high-contrast coefficients. Solutions to these model problems are multiscale in nature. The solutions to these problems are often approximated using the Finite Element Method (FEM), Multiscale Finite Element Method (MsFEM) or alternative forms of these, (cf., \citet{yang1997least,wihler2004locking,MR2477579,gatica2009augmented,di2013locking,xia2014mib}. and references therein). 

Herein, we study linear elasticity problems in heterogeneous media. Our goal is to devise approximate solutions that account for the high contrast in the coefficients. We focus on the dependence of the contrast in the coefficients where the contrast is referred to as the ratio of the jumps in the physical properties. We follow 
the analysis  presented in \citet{calo2014asymptotic} that consists of  deriving an asymptotic expansion for the solution of the elliptic differential equation in heterogeneous media. Thus, we derive asymptotic expansions to solve linear elasticity problems with high contrast.

The linear elasticity equations model the equilibrium and the local strain of deformable bodies; see  \citet{MR936420,MR1477663,MR0010851,MR0075755,
kang1996mathematical}.
 The constitutive laws relating stresses and strains  depend on the material and the process modeled. For composite materials, physical properties such as the Young's modulus can vary several orders of magnitude and we seek to understand the effects of these variations on the solution. In this setting, the asymptotic expansions that express the solution are useful tools to understand the effects of the high contrast and the interactions between different materials.

We consider the equilibrium equations for a linear elastic material in a smooth domain $D\subset \mathbb{R}^d$. Given $u\in H^1(D)^d$ that represents the displacement field, we denote
\[
\epsilon =\epsilon (u)= \left[ \epsilon_{ij} =\frac{1}{2} \left(\frac{\partial u_i}{\partial x_j}+ \frac{\partial u_j}{\partial x_i} \right) \right],
\]
where $\epsilon$ is the strain tensor which linearly depends on the derivatives of the displacement field $u$ (see \citet{gonzalez2008first,malvern1969introduction}). We also introduce the stress tensor $\tau(u)$, which depends on the value of strains and is defined as
\begin{equation} \label{tensortau}
\tau =\tau(u)=2\mu\epsilon(u)+\lambda\mbox{tr}\epsilon(u) I_{d\times d},
\end{equation}
where $I_{d\times d}$ is the identity matrix in $\mathbb{R}^d$ and $\mbox{tr}\epsilon(u)=\mbox{div}(u)$. The Lam\'e coefficients $\lambda$ and $\mu$ describe the elastic response of an isotropic material, see, e.g., \citet{gonzalez2008first,kang1996mathematical}.

We assume that the Poisson ratio $\nu=\tfrac{\lambda}{2(\lambda+\mu)}$  is bounded away from $0.5$, i.e., the Poisson ratio satisfies $0<\nu\leq \nu_0<0.5$ for some constant value 
$\nu_0$. The volumetric strain modulus is given by
\[
K=\frac{E}{3(1-2\nu)}>0,
\]
and thus $1-2\nu>0$ (see \citet{kang1996mathematical}). Given these assumptions, then $\nu=\nu(x)$ can only have mild variations in $D$.

We introduce the heterogeneous function $E=E(x)$ that represents the Young's modulus and thus express the shear modulus as
\[
\mu(x)=\frac{1}{2}\frac{E(x)}{1+\nu(x)}=\tilde{\mu}(x)E(x),
\]
where  $\tilde{\mu}=1/2(1+\nu)$. Thus,
\[
\lambda(x)=\frac{1}{2}\frac{E(x)\nu(x)}{(1+\nu(x))(1-2\nu(x))}=\tilde{\lambda}(x)E(x),
\]
where we use $\tilde{\lambda}=\tfrac{\nu}{2(1+\nu)(1-2\nu)}$. The spatial variation of $E$ drives the multiscale response of the solution. We denote
\[
\tilde{\tau}(u)=2\tilde{\mu}\epsilon(u)+\tilde{\lambda}\mbox{tr} \epsilon(u) I_{d\times d}.
\]

Given a vector field $f$ we consider the problem
\begin{equation} \label{eq:1}
-\mbox{div} (\tau(u))=f, \quad \mbox{in }D, 
\end{equation}
with $u=g$ on $\partial D$. The tensor $\tau$ is defined in (\ref{tensortau}). We analyze in detail a binary medium $E(x)$ with elastic background and one inclusion (a stiff body) for the case of one inelastic inclusion. The analysis of the case with several highly inelastic inclusions is similar. To parametrize the problem, we consider the background with stiffness $1$ and the inclusions with a relative stiffness denoted by $\eta$. We derive expansions of the form (see \citet{calo2014asymptotic})
\begin{equation}
u_\eta =u_0+\frac{1}{\eta}u_1+\frac{1}{\eta^2}u_2+\cdots.
\end{equation}
We define each term in the expansion using local problems. We then study the convergence of the expansion in the $H^1$ norm.

The rest of the paper is organized as follows. In Section \ref{Sec2:ProblemSetting} we recall the weak formulation and provide an overview of the derivation of 
the expansion for high contrast inclusions. In Section \ref{Sec3:ExpansionInelastic}, the convergence for this asymptotic expansion is described. Finally, in Section \ref{sec4:Conclusions} we state our conclusions and final comments.

\section{One interior inclusion problem: problem statement}
\label{Sec2:ProblemSetting}

Let $D\subset \mathbb{R}^d$ be a polygonal domain or a domain with smooth boundary. We consider the following weak formulation of (\ref{eq:1}). Find $u\in H^1(D)^d$ such that
\begin{equation}
\label{eq:7}
\left\{\begin{array}{ll}
 {\mathcal A}(u,v)={\cal F}(v), & \mbox{for all } v\in H_0^1(D)^d,\\
\hspace{.4in} u=g, &\mbox{ on } \partial D,
\end{array}\right.
\end{equation}
where the bilinear form ${\mathcal A}$ and the linear functional $\mathcal F$ are defined by
\begin{equation}
\label{eq:8}
\begin{array}{ll}
{\mathcal A}(u,v)=\displaystyle \int_{D}2\tilde{\mu}E\epsilon(u)\cdot \epsilon(v)+\tilde{\lambda}E\mbox{tr}\epsilon(u)\mbox{tr}\epsilon(v),  
 \mbox{ for all }  u,v\in H_0^1(D)^d,\\
\end{array}
\end{equation}
and
\begin{equation}
\begin{array}{ll}
 {\cal F}(v)=\displaystyle\int_Dfv, \mbox{ for all } v\in H_0^1(D)^d,
\end{array}
\end{equation}
respectively, with $\epsilon (u)\cdot \epsilon (v) := \sum_{i,j=1}^{d}\epsilon_{ij}(u)\epsilon_{ij}(v)$.
\begin{figure}[!h]
\centering
\includegraphics[width=0.7\textwidth]{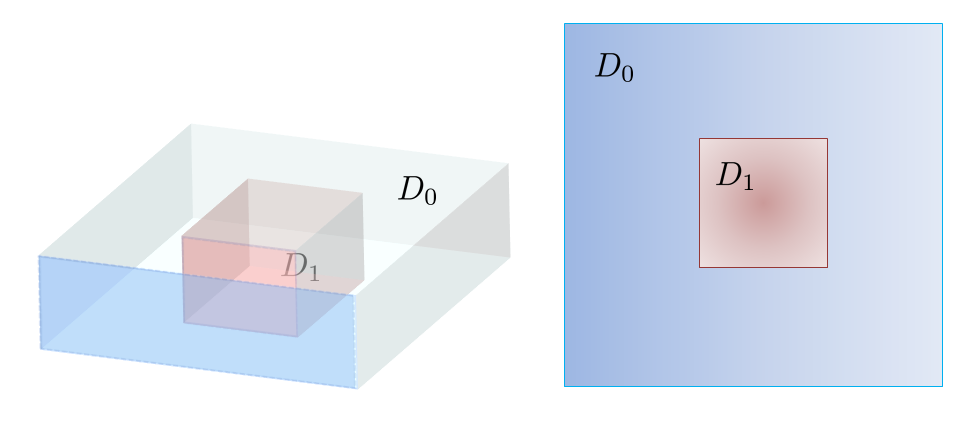}
\caption{Geometric configuration with one interior inclusion.}
\label{Fig1:Geometrycong} 
\end{figure}

The domain $D$ is the disjoint union of a background domain and one inclusion, that is, $D=D_0\cup \overline{D}_1$. We assume that $D_0$ and $D_1$ are polygonal domains or domains with smooth boundaries. Let $D_0$ represent the background domain and the sub-domain $D_1$ represent inclusion. For simplicity of the presentation we consider only one interior inclusion. Given $w\in H^1(D)^d$ we use the notation $w^{(m)}$, for the restriction of $w$ to the domain $D_m$, that is
\[
w^{(m)}=w|_{D_m}, \quad m=0,1.
\]
We also introduce the following notation, given $\Omega \subset D$, we denote by ${\mathcal A}_\Omega$ the bilinear form
\[
{\mathcal A}_\Omega (u,v)=\int_\Omega 2\tilde{\mu}\epsilon(u)\cdot \epsilon(v)+\tilde{\lambda}\mbox{tr}\epsilon(u)\mbox{tr}\epsilon(v), 
\]
defined for functions in $u,v\in H^1(\Omega )^d$. If $\Omega \subseteq D_m$, then ${\mathcal A}_\Omega$ does not depend on the Young's modulus $E(x)$, since $E(x)$ is assumed to be defined by piecewise constants. We denote by ${\mathcal RB}(\Omega)$ the subset of \emph{rigid body motions} defined on $\Omega$, for instance, if $d=2$ we have that
\begin{equation}\label{RB2}
{\mathcal RB}(\Omega)=\{(a_1,a_2)+b(x_2,-x_1): \ a_1,a_2,b \in \mathbb{R} \},
\end{equation}
or if $d=3$ we have
\begin{equation}\label{RB3}
{\mathcal RB}(\Omega)=\{(a_1,a_2,a_3)+(b_1,b_2,b_3)\times (x_1,x_2,x_3): \ a_i,b_i \in \mathbb{R} , i=1,2,3 \}.
\end{equation}

\section{One interior inclusion problem: series expansion}
\label{Sec3:ExpansionInelastic}

We derive and analyze the asymptotic expansion for the case of a single highly elastic inclusion.
We follow \cite{calo2014asymptotic}.

\subsection{Derivation}
\label{subsec:3.1}

Let $E$ be defined by
\begin{equation} \label{eq:11}
E(x)=\left\{\begin{array}{ll}
\eta, & x \in D_1,\\
1, & x \in D_0=D\setminus \overline{D}_1,
\end{array}\right.
\end{equation}
and denote by $u_\eta$ the solution of the weak formulation (\ref{eq:7}). We assume that $D_1$ is compactly included in $D$ ($\overline{D}_1 \subset D$). Since $u_\eta$ is the solution of (\ref{eq:7}) with the coefficient (\ref{eq:11}), we have
\begin{equation}\label{eq:12}
{\mathcal A}_{D_0}(u_\eta,v)+\eta {\mathcal A}_{D_1}(u_\eta,v)={\cal F}(v),\quad  \mbox{for all } v \in H_0^1(D).
\end{equation}
We seek to determine $\{u_j\}_{j=0}^\infty \subset H^1(D)^d$ such that 
\begin{equation} \label{eq:13}
u_\eta =u_0+\frac{1}{\eta}u_1+\frac{1}{\eta ^2}u_2+\cdots =\sum_{j=0}^\infty \eta^{-j}u_j,
\end{equation}
and such that they satisfy the following Dirichlet boundary conditions
\begin{equation}\label{eq:14}
u_0=g \mbox{ on } \partial D \quad \mbox{and } \quad u_j=0 \mbox{ on } \partial D \mbox{ for } j\geq 1. 
\end{equation}
We substitute (\ref{eq:13}) into (\ref{eq:12}) to obtain that for all $v\in H_0^1(D)$ we have
\begin{equation} \label{eq:15new}
\eta{\mathcal A}_{D_1}(u_0,v)+\sum_{j=0}^\infty \eta^{-j}\Big({\mathcal A}_{D_0}(u_j,v)+{\mathcal A}_{D_1}(u_{j+1},v)\Big)={\cal F}(v).
\end{equation}
Now we collect terms with equal powers of $\eta$ and analyze the resulting sub-domain equations.

\subsubsection{Term corresponding to $\eta^1$}

In (\ref{eq:15new}) there is one term corresponding to $\eta$ to the power $1$, thus we obtain the following equation
\begin{equation} \label{eq:16}
{\cal A}_{D_1}(u_0,v)=0 \mbox{ for all }v\in H_0^1(D)^d.
\end{equation}
The problem above corresponds to an elasticity equation posed on $D_1$ with homogeneous Neumann boundary conditions. Since we assume that $\overline{D}_1\subset D$, we conclude that $u_0^{(1)}$ is a \emph{rigid body motion}, that is, $u_0^{(1)}\in {\cal RB}(D_1)$ where ${\cal RB}$ is defined above in (\ref{RB2}) and (\ref{RB3}).

In the general case, the meaning of this equation depends on the relative position of the inclusion $D_1$ with respect to the boundary and thus may need to take the boundary data into account.

\subsubsection{Terms corresponding to $\eta^0=1$}

The equation (\ref{eq:14}) contains three terms corresponding to $\eta$ to the power $0$, which are
\begin{equation} \label{eq:17new}
{\cal A}_{D_0}(u_0,v)+{\cal A}_{D_1}(u_1,v)={\cal F}(v), \mbox{ for all }v\in H_0^1(D)^d.
\end{equation}
Let
\[
V_{\cal RB}=\{ v\in H_0^1(D)^d, \mbox{ such that } v^{(1)}=v|_{D_1} \in {\cal RB}(D_1)\}.
\]
If we consider  $z\in V_{\cal RB}$ in equation (\ref{eq:17new}) we conclude that $u_0$ satifies the following problem
\begin{align}
\label{eq:18new}
{\cal A}_{D_0}(u_0,z)&= {\cal F}(z),&\mbox{for all }  z\in V_{\cal RB},
\end{align}
with $u_0=g$ on  $\partial D$. The problem (\ref{eq:18new}) is elliptic and has a unique solution (for details see \citet{MR1477663}). To analyze this problem further we proceed as follows. Let $\{ \xi_{D_1;\ell}\}_{\ell=1}^{L_d}$ be a basis for the ${\cal RB} (D_1)$ space, where $L_d$ is the dimension of the space ${\cal RB}$, that is, $L_d=3$ for $2D$ problems and $L_d=6$ for $3D$ ones. Then we have that $u_0^{(1)}=\sum_{\ell=1}^{L_d}c_{0;\ell}\xi_{D_1;\ell}$. We define the harmonic extension of the rigid body motions, $\chi_{D_1;\ell}\in H_0^1(D)^d$ such that
\[
\chi_{D_1;\ell}^{(1)}=\xi_{D_1;\ell},\quad \mbox{ in } D_1,
\]
while the harmonic extension of its boundary data in $D_0$ is given by
\begin{eqnarray*}
\label{eq:20}
{\cal A}_{D_0}(\chi _{D_1;\ell}^{(0)},z)&=0, &\mbox{ for all }  z\in H_0^1(D_0)^d,\\
\chi_{D_1;\ell}^{(0)}&=\xi_{D_1;\ell}, &\mbox{ on } \partial D_1,\\
\chi_{D_1;\ell}^{(0)}&=0, &\mbox{ on } \partial D.
\end{eqnarray*}

\begin{rmk} \label{remark1}
Let $w$ be a harmonic extension to $D_0$ of its Neumann data on $\partial D_0$. That is, $w$ satisfies the following problem
\[
{\cal A}_{D_0}(w,v)=\int_{\partial D_0}\widetilde{\tau}(w)\cdot n_0v \quad \mbox{ for all }v \in H^1(D_0)^d,
\]
with boundary data $\tilde{\tau}(w)\cdot n_0$ on $\partial D_0$. Since $\chi_{D_1;\ell}=0$ in $\partial D$  and $\chi_{D_1;\ell}=\xi_{D_1;\ell}$ on $\partial D_1$ with $\ell=1,\dots ,L_d$. We readily have that
\[
{\cal A}_{D_0}(w,\chi_{D_1;\ell})=0\left( \int_{\partial D}\widetilde{\tau}(w)\cdot n_1 \right)+\left( \int_{\partial D_1}\widetilde{\tau}(w)\cdot n_0\xi_{D_1;\ell}\right),
\]
and we conclude that for every harmonic function on $D_0$
\begin{equation}\label{eq:21}
{\cal A}_{D_0}(w,\chi_{D_1;\ell})=\int_{\partial D_1}\widetilde{\tau}(w)\cdot n_0\xi_{D_1;\ell}.
\end{equation}
In particular, taking $w=\chi_{D_1;\ell}$ we have
\begin{equation}\label{eq:22}
{\cal A}_{D_0}(\chi_{D_1;\ell},\chi_{D_1;\ell})=\int_{\partial D_1} \widetilde{\tau}(\chi_{D_1;\ell})\cdot n_0\chi_{D_1;\ell}.
\end{equation}
\end{rmk}

To obtain an explicit formula for $u_0$ we use the fact that problem (\ref{eq:18new}) is elliptic and has a unique solution, and the property of the harmonic characteristic functions described in the Remark \ref{remark1}. Thus, We can decompose $u_0$ into the harmonic extension of its value in $D_1$, given by $ u_0^{(1)}=\sum_{\ell=1}^{L_d}c_{0;\ell}\xi_{D_1;\ell}$, plus the remainder $u_{0,0}\in H^1(D_0)^d$. Thus, we write
\begin{equation}\label{eq:23}
u_0=u_{0,0}+\sum_{\ell=1}^{L_d}c_{0;\ell}\xi_{D_1;\ell},
\end{equation}
where $u_{0,0}\in H^1(D)$ is defined by $u_{0,0}^{(1)}=0$ in $D_1$ and $u_{0,0}^{(0)}$ solves the following Dirichlet problem
\begin{eqnarray}\label{eq:24}
{\cal A}_{D_0}(u_{0,0}^{(0)},y)&=&{\cal F}(y),  \mbox{ for all }y\in H_0^1(D_0)^d,\\\nonumber
u_{0,0}^{(0)}&=&0,  \mbox{ on }\partial D_1,\\
u_{0,0}^{(0)}&=&g,  \mbox{ on }\partial D.\nonumber
\end{eqnarray}
From (\ref{eq:18new}) and  (\ref{eq:23}) we get that 
\begin{equation}\label{eq:25}
\sum_{\ell=1}^{L_d}c_{0;\ell}{\cal A}_{D_0}(\chi_{D_1;\ell},\chi_{D_1;m})={\cal F}(\chi_{D_1;m})-{\cal A}_{D_0}(u_{0,0},\chi_{D_1;m}),
\end{equation}
with $m=1,\dots ,L_d$.  From (\ref{eq:25}) we obtain the constants $c_{0;\ell}$, $\ell=1,\dots,L_d$ by solving a $L_d\times L_d$ linear system. As readily seen, the $L_d\times L_d$ matrix 
\begin{equation}\label{matrixa_lm}
\mathbf{A}_{geom}=\left[ a_{\ell m}\right]_{\ell, m=1}^{L_d},\quad\mbox{where }a_{\ell m}={\cal A}_{D_0}(\chi_{D_1;\ell},\chi_{D_1;m}).
\end{equation}
The matrix $\mathbf{A}_{geom}$ is positive. Given the explicit form of $u_0$, we use (\ref{eq:23}) in (\ref{eq:17new}) to find $u_1^{(1)}=u_1|_{D_1}$ from the analysis of (\ref{eq:17new}) we conclude that $u_0^{(0)}$ satisfies the local Dirichlet problem
\begin{equation*}
{\cal A}_{D_0}(u_0^{(0)},z)=\int_{D_0}fz,\mbox{ for all }z\in H_0^1(D_0)^d,
\end{equation*}
with given  boundary data $\partial D_0$ in \eqref{eq:24}. Equation (\ref{eq:17new}) also represents the transmission conditions across $\partial D_1$ for the functions $u_0^{(0)}$ and $u_1^{(1)}$. This is easier to see when the forcing $f$ is square integrable. From now on, in order to simplify the presentation, we assume that $f\in L^2(D)$. If $f\in L^2(D)$ , we have that $u_0^{(0)}$ and $u_1^{(1)}$ are the only solutions of the problems
\[
{\cal A}_{D_0}(u_0^{(0)},z)=\int_{D_0}fz+\int_{\partial D_0\setminus \partial D}\widetilde{\tau}(u_0^{(0)})\cdot n_0z, \quad \mbox{ for all }z\in H^1(D_0)^d, 
\]
with $z=0$ on $\partial D$ and $u_0^{(0)}=g$ on $\partial D$, and
\[
{\cal A}_{D_1}(u_1^{(1)},z)=\int_{D_1}fz+\int_{\partial D_1}\widetilde{\tau}(u_1^{(1)})\cdot n_1z, \quad \mbox{ for all }z\in H^1(D_1)^d. 
\]
Replacing these last two equations back into (\ref{eq:17new}) we conclude that
\begin{equation}\label{eq:26}
\widetilde{\tau}(u_1^{(1)})\cdot n_1=-\widetilde{\tau}(u_0^{(0)})\cdot n_0, \quad \mbox{ on }\partial D_1.
\end{equation}
Using this interface condition we can obtain $u_1^{(1)}$ in $D_1$ by writing
\begin{equation}\label{eq:27-1}
u_1^{(1)}=\widetilde{u}_1^{(1)}+\sum_{\ell=1}^{L_d}c_{1;\ell}\xi_{D_1;\ell},
\end{equation}
where $\widetilde{u}_1^{(1)}$ solves the Neumann problem
\begin{equation} \label{eq:27}
{\cal A}_{D_1}(\widetilde{u}_1^{(1)},z)=\int_{D_1}fz-\int_ {\partial D_1}\widetilde{\tau}(u_0^{(0)})\cdot n_1z, \quad \mbox{for all }z\in H^1(D_1)^d.
\end{equation}
where the constants $c_{1;\ell}$ are chosen later. Problem (\ref{eq:27}) needs the following compatibility conditions
\[
\int_{D_1}f\xi+\int_{\partial D_1}\widetilde{\tau}(u_0^{(0)})\cdot n_1\xi =0, \quad \mbox{for all }\xi \in {\cal RB},
\]
which, using (\ref{eq:23}) and (\ref{eq:26}) and noting that $\chi_{D_1;\ell}$ in $D_1$, reduces to
\begin{equation} \label{eq:28}
\sum_{\ell =1}^{L_d}c_{0;\ell}\int_{\partial D_1}\widetilde{\tau}(\chi_{D_1;\ell})\cdot n_1\chi_{D_1;m}=\int_{D_1}f\chi_{D_1;m}-\int_{\partial D_1}\widetilde{\tau}(\widetilde{u}_{0,0})\cdot n_1\chi_{D_1;m} 
\end{equation}
for $m=1,\dots ,L_d$. This system of $L_d$ equations is the same encountered before in (\ref{eq:25}). The fact that the two systems are the same follows from the next two integration by parts relations:
\begin{enumerate}
\item [(i)] according to Remark \ref{remark1}
\begin{equation} \label{eq:29}
\int_{\partial D_1}\widetilde{\tau}(\chi_{D_1;\ell})\cdot n_1\chi_{D_1;m}={\cal A}_{D_0}(\chi_{D_1;\ell},\chi_{D_1;m}).
\end{equation}
\item [(ii)] we have
\begin{equation} \label{eq:30}
\int_{\partial D_1}\widetilde{\tau}\left(\widetilde{u}_{0,0}\cdot n_1\chi_{D_1;m}\right)={\cal A}_{D_0}(u_{0,0},\chi_{D_1;m})-\int_{D_0}f\chi_{D_1;m}.
\end{equation}
\end{enumerate}
By replacing the relations in \eqref{eq:29} and \eqref{eq:30} into (\ref{eq:28}) we obtain (\ref{eq:25}) and conclude that the compatibility condition of problem (\ref{eq:27}) is satisfied. Next, we discuss how to compute $u_1^{(0)}$ and $\widetilde{u}_1^{(0)}$ to completely define the functions $u_1\in H^1(D)^d$  and $\widetilde{u}_1\in H^1(D)^d$. These are presented for general $j\geq 1$ since the construction is independent of $j$ in this range.

\subsubsection{Term corresponding to $\eta^{-j}$ with $j\geq 1$}
For powers $1/\eta$ larger or equal to one there are only two terms in the summation that lead to the following system
\begin{equation}\label{eq:31}
{\cal A}_{D_0}(u_j,v)+{\cal A}_{D_1}(u_{j+1},v)=0, \quad \mbox{for all }v\in H_0^1(D)^d.
\end{equation}
This equation represents both the sub-domain problems and the transmission conditions across $\partial D_1$ for $u_j^{(0)}$ and $u_{j+1}^{(1)}$. Following a similar argument to the one given above, we conclude that $u_j^{(0)}$ is harmonic in $D_0$ for all $j\geq 1$ and that $u_{j+1}^{(1)}$ is harmonic in $D_1$ for $j\geq 2$. As before, we have
\begin{equation}\label{eq:32}
\widetilde{\tau}(u_{j+1}^{(1)})\cdot n_1=-\tau (u_j^{(0)})\cdot n_0.
\end{equation}
Since $u_j^{(1)}$ in $D_1$, (e.g., $u_1^{(1)}$ above) is given by the solution of a Neumann problem in $D_1$. The solution of a Neumann linear elasticity problem is defined up to a rigid body motion. To uniquely determine $u_j^{(1)}$, we write
\begin{equation} \label{eq:33}
u_j^{(1)}=\widetilde{u}_j^{(1)}+\sum_{\ell =1}^{L_d}c_{j;\ell}\xi_{D_1;\ell},
\end{equation}
where $u_j^{(1)}$ is $L^2$-orthogonal to the rigid body motion of $D_1$ and the appropriate $c_{j;\ell}$ is determined below.
Given $u_j^{(1)}$ in $D_1$ we find $u_j^{(0)}$ in $D_0$ by solving a Dirichlet problem with known Dirichlet data, that is,
\begin{equation}\label{eq:34}
\begin{array}{l}
\mathcal{A}_{D_0}(u_{j}^{(0)}, z)=0 \mbox{ for all } 
z\in H^1_0(D_0)^d \\
u_j^{(0)}=u_j^{(1)} \ \left(=\widetilde{u}_j^{(1)}+\displaystyle \sum_{\ell=1}^{L_d}
c_{j;\ell}\xi_{D_1;\ell}\right) \mbox{ on } 
\partial D_1
\quad \mbox{ and } \quad u_j=0 \mbox{ on } \partial D.
\end{array} 
\end{equation}
We conclude that
\begin{equation} \label{eq:35}
u_j=\widetilde{u_j}+\sum_{\ell =1}^{L_d}c_{j;\ell}\chi_{D_1;\ell},
\end{equation}
where $\widetilde{u}_{j}^{(0)}$ is defined by (\ref{eq:34}) replacing $c_{j;\ell}$ by $0$. This completes the construction of $u_j$.
Now we proceed to show how to find $u_{j+1}^{(1)}$ in $D_1$. For this, we use (\ref{eq:30}) and (\ref{eq:31}) which lead to the following Neumann problem
\begin{equation} \label{eq:36}
{\cal A}_{D_1}(\widetilde{u}_{j+1}^{(1)},z)=-\int_{\partial D_1}\widetilde{\tau}(u_{j}^{(0)})\cdot n_0z \quad \mbox{for all }z\in H^1(D_1)^d.
\end{equation}
The compatibility condition for this Neumann problem is satisfied if we choose $c_{j;\ell}$ the solution of the $L_d\times L_d$ system 
\begin{equation} \label{eq:37}
\sum_{\ell =1}^{L_d}c_{j;\ell}\int_{\partial D_1}\widetilde{\tau}(\chi_{D_1;\ell}\cdot n_1\chi_{D_1;m}=-\int_{\partial D_1}\widetilde{\tau}(\widetilde{u}_j)\cdot n_1\chi_{D_1;m},
\end{equation}
with $m=1,\dots ,L_d$. As pointed out before, see (\ref{eq:29}), this system can be written as
\[
\sum_{\ell =1}^{L_d}c_{j;\ell}{\cal A}_{D_0}(\chi_{D_1;\ell},\chi_{D_1;m})=-{\cal A}_{D_0}(\widetilde{u}_j,\chi_{D_1;m}).
\]
In this form we readily see that this $L_d\times L_d$ system matrix is positive definite and therefore solvable.
We can choose $u_{j+1}^{(1)}$ in $D_1$ such that
\[
u_{j+1}^{(1)}=\widetilde{u}_{j+1}^{(1)}+\sum_{\ell =1}^{L_d}c_{j+1;\ell}\xi_{D_1;\ell},
\]
where $\widetilde{u}_{j+1}^{(1)}$ is properly chosen and, as before
\[
\sum_{\ell =1}^{L_d}c_{j+1,\ell}\int_{\partial D_1}\widetilde{\tau}(\chi_{D_1;\ell})\cdot \chi_{D_1;m}=-\int_{\partial D_1}\widetilde{u}_{j+1}\cdot n_1\chi_ {D_1;m},
\]
$m=1,\dots , L_d$. Therefore we have the compatibility condition of the Neumann problem to compute $u_{j+2}^{(1)}$. See the equation (\ref{eq:36}).

\subsection{Convergence in $H^1(D)^d$}\label{convergenceelasti}

We study the convergence of the expansion (\ref{eq:13}) with the Dirichlet data (\ref{eq:14}). We assume that $\partial D$ and $\partial D_1$ are sufficiently smooth. We follow the analysis introduced in \citet{calo2014asymptotic}.

We use standard Sobolev spaces (see for instance, \citet{Adams-book2003}). Given a sub-domain $D$, we use the $H^1(D)^d$ norm given by
\[
\| v\|_{H^1(D)^d}^2=\|v\|_{L^2(D)^d}^2+\|\nabla v\|_{L^2(D)^d}^2, 
\]
and the seminorm
\[
|v|_{H^1(D)^d}^2=\|\nabla v\|_{L^2(D)^d}^2.
\]
We also use standard trace spaces $H^{1/2}(\partial D)^d$ and the dual space $H^{-1}(D)^d$.

\begin{lem}\label{mainlemma}
Let $\widetilde{w}\in H^1(D)$ be harmonic in $D_0$ and define 
\[
w=\widetilde{w}+\sum_{\ell =1}^{L_d}c_{0;\ell}\xi_{D_1;\ell},
\]
where $Y=(c_{0;1},\dots ,c_{0;L_d})$ is the solution of the $L_d$-dimensional linear system
\begin{equation}\label{Ld system}
\mathbf{A}_{geom} Y=-W,
\end{equation}
where
\[
\mathbf{A}_{geom}=\left[a_{\ell m}\right]_{\ell ,m=1}^{L_d},\quad\mbox{with }a_{\ell ,m}={\cal A}_{D_0}(\xi_{D_1;\ell},\xi_{D_1;m}).
\]
and
\begin{equation}
W=({\cal A}_{D_0}(\widetilde{w},\xi_{D_1;1}),\dots, {\cal A}_{D_0}(\widetilde{w},\xi_{D_1;L_d})),
\end{equation}
we recall that $L_d=1,\dots ,\ell$ is the spatial dimension of the subset of rigid body motions ${\cal RB}$. Then,
\[
\|w\|_{H^1(D)^d}\preceq\|\widetilde{w}\|_{H^1(D)^d},
\]
where the hidden constant is the Korn inequality constant of $D$.
\end{lem}

\begin{proof}
Since $\sum_{\ell =1}^{L_d}c_{0;\ell}\xi_{D_{1};\ell}$ is the Galerkin projection of $\widetilde{w}$ into the space $\mbox{Span} \{ \xi_{D_1;\ell}\}_{\ell =1}^{L_d}$. From the analysis of Galerkin formulations, we have
\begin{eqnarray*}
{\cal A}_{D_0}\left(\sum_{\ell=1}^{L_d}c_{0;\ell}\xi_{D_1;\ell},\sum_{\ell=1}^{L_d}c_{0;\ell}\xi_{D_1;\ell}\right) & = & Y^{T}{\bf A}_{geom}Y=-Y^{T}W\\
& = & -\sum_{\ell=1}^{L_d}c_{0;\ell}{\cal A}_{D_0}(\widetilde{w},\xi_{D_1;\ell})\\
& = & -{\cal A}_{D_0}\left(\widetilde{w},\sum_{\ell=1}^{L_d}c_{0;\ell}\xi_{D_1;\ell}\right)\\
& \leq & |\widetilde{w}|_{H^1(D)^d}\left|\sum_{\ell=1}^{L_d}c_{0;\ell}\xi_{D_1;\ell}\right|_{H^1(D_0)^d},
\end{eqnarray*}
by the Korn inequality
\begin{eqnarray*}
\left| \sum_{\ell =1}^{L_d}c_{0;\ell}\xi_{D_1;\ell} \right|_{H^1(D_0)^d}^2 & \leq  & C{\cal A}_{D_0}\left(\sum_{0;\ell}^{L_d}c_{0;\ell}\xi_{D_1;\ell},\sum_{0;\ell}^{L_d}c_{0;\ell}\xi_{D_1;\ell}\right)\\
& \leq & |\widetilde{w}|_{H^1(D_0)^d}\left|\sum_{\ell =1}^{L_d}C_{0;\ell}\xi_{D_1;\ell}\right|_{H^1(D_0)^d},
\end{eqnarray*}
so
\[
\left|\sum_{\ell=1}^{L_d}c_{0;\ell}\xi_{D_1;\ell}\right|_{H^1(D_0)^d}\leq |\widetilde{w}|_{H^1(D_0)^d}.
\]
Using the fact above, we get
\[
\|w\|_{H^1(D)^d}\leq \| \widetilde{w}\|_{H^1(D)^d}+\left\Vert \sum_{\ell=1}^{L_d}c_{0;\ell}\xi_{D_1;\ell}\right\Vert _{H^1(D)^d}\preceq \| \widetilde{w}\|_{H^1(D)^d}.
\]
\end{proof}

For the proof of the convergence of the expansion (\ref{eq:13}) with the boundary condition (\ref{eq:14}), we consider the following additional results obtained by applying the Lax-Milgram theorem in \citet{Brezis-book2010} and the trace theorem in \citet{Adams-book2003}.   

\begin{lem}\label{lemmaconver1}
Let $u_0$ in (\ref{eq:23}), with $u_{0,0}$ defined in (\ref{eq:24}), and $u_1$ be defined by (\ref{eq:27}) and (\ref{eq:34}) with $j=1$. We have that
\begin{equation} \label{eq:4.35}
\|u_0\|_{H^1(D)^d}\preceq  \|f\|_{H^{-1}(D)^d}+\|g\|_{H^{1/2}(\partial D)^d},
\end{equation}
\begin{equation}\label{eq:4.36}
\|\widetilde{u}_1\|_{H^1(D_1)^d} \preceq  \|f\|_{H^{-1}(D_1)^d}+\|g\|_{H^{1/2}(\partial D)^d}
\end{equation}
and
\begin{equation}\label{eq:4.37}
\|\widetilde{u}_1\|_{H^1(D_0)^d}\preceq  \|\widetilde{u}_1\|_{H^{1/2}(\partial D_1)^d} \preceq \|\widetilde{u}_1\|_{H^{1}( D_1)^d}.
\end{equation}
\end{lem}
 
\begin{proof}
From the definition of $u_{0,0}$ in (\ref{eq:24}) we have that
\[
\|u_{0,0}\|_{H^1(D_0)^d}\preceq\|f\|_{H^{-1}(D_0)^d}+\|g\|_{H^{1/2}(\partial D)^d},
\]
for more details see for instance, \citet{Adams-book2003}.

Now, using the Korn inequality for Dirichlet data and the Lax-Milgram theorem we have
\[
\|u_0\|_{H^1(D)^d}\preceq |u_0|_{H^1(D)^d}=|u_0|_{H^1(D_0)^d}\leq |u_{0,0}|_{H^1(D_0)^d}+\left|\sum_{\ell=1}^{L_d}c_{0;\ell}\xi_{D_1;\ell}\right|_{H^1(D_0)^d},
\]
and using a similar argument  to the one used in Lemma \ref{mainlemma}, we have that
\begin{eqnarray}\nonumber
\left|\sum_{\ell=1}^{L_d}c_{0;\ell}\xi_{D_1;\ell}\right|_{H^1(D_0)^d}^2 & \leq & C\left|u_{0,0}\right|_{H^1(D_0)^d}\left|\sum_{\ell=1}^{L_d}c_{0;\ell}\xi_{D_1;\ell}\right|_{H^1(D_0)^d}+\int_{D} f\left(\sum_{\ell=1}^{L_d}c_{0;\ell}\xi_{D_1;\ell}\right)\\\nonumber
& \preceq & \left|u_{0,0}\right|_{H^1(D_0)^d}\left|\sum_{\ell=1}^{L_d}c_{0;\ell}\xi_{D_1;\ell}\right|_{H^1(D_0)^d}+\|f\|_{H^{-1}(D)^d}\left|\sum_{\ell=1}^{L_d}c_{0;\ell}\xi_{D_1;\ell}\right|_{H^1(D_0)^d},\nonumber
\end{eqnarray}
so
\begin{equation*}
\left|\sum_{\ell=1}^{L_d}c_{0;\ell}\xi_{D_1;\ell}\right|_{H^1(D_0)^d}\preceq \left|u_{0,0}\right|_{H^1(D_0)^d}+\|f\|_{H^{-1}(D)^d}.
\end{equation*}
Using this fact and the definition of $u_{0,0}$, we conclude that
\[
\|u_0\|_{H^1(D)^d}\preceq\|f\|_{H^{-1}(D)^d}+\|g\|_{H^{1/2}(\partial D)^d}.
\]
This concludes the proof of \eqref{eq:4.35}.

Equation (\ref{eq:4.36}) uses a similar argument to the one used in the above proof, we use the Korn inequality for Neumann conditions and the trace theorem. Finally, the equation (\ref{eq:4.37}) is obtained using Korn inequality for Dirichlet conditions and the trace theorem. Details are not included for the sake of brevity.
\end{proof}

\begin{lem}\label{lemmaconver2}
Let $u_j$ defined on $D_0$ by (\ref{eq:34}) with $c_{j;\ell}$ and $u_{j+1}$ defined on $D_1$ by (\ref{eq:36}). For $j\geq 1$ we have that
\[
\|u_{j+1}\|_{H^1(D)^d}\preceq \|u_j\|_{H^1(D_0)^d}.
\] 
\end{lem}

\begin{proof}
Let $j\geq 1$. Consider $\widetilde{u}_{j+1}$ defined by the Dirichlet in (\ref{eq:34}). From the Lemma \ref{mainlemma} %
and combining the Korn inequality for Dirichlet conditions and the trace theorem, 
we have
\[
\|u_{j+1}\|_{H^1(D)^d}\preceq \|\widetilde{u}_{j+1}\|_{H^1(D)^d}\leq C\|\widetilde{u}_{j+1}\|_{H^1(D_1)^d},
\]
applying the Korn inequality for the Dirichlet conditions in the last equation we obtain
\[
\|\widetilde{u}_{j+1}\|_{H^1(D_1)^d}\preceq \|u_j\|_{H^1(D_0)^d}.
\]
Combining these inequalities we have
\[
\|u_{j+1}\|_{H^1(D)^d}\preceq \|u_j\|_{H^1(D_0)^d}.
\]
This concludes the proof.
\end{proof}

\begin{thm}
There is a constant $C>0$ such that for every $\eta>C$, the expansion (\ref{eq:13}) converges (absolutely) in $H^1(D)$. The asymptotic limit $u_0$ satisfies problem (\ref{eq:18new}) and $u_0$ can be computed using formula (\ref{eq:23}).
\end{thm}

\begin{proof}
From the Lemma \ref{lemmaconver2} applied repeatedly $j-1$ times, we get that for every $j\geq 2$ there is a constant $C$ such that
\begin{eqnarray*}
\|u_j\|_{H^1(D)^d} & \leq & C\|u_{j-1}\|_{H^1(D_0)^d}\leq C\|u_{j-1}\|_{H^1(D)^d}\\
& \leq & \cdots \leq C^{j-1}\|\widetilde{u}_1\|_{H^1(D_0)^d}
\end{eqnarray*}
and then
\[
\left\Vert \sum_{j=2}^{\infty}\eta^{-j}u_j\right\Vert_{H^1(D)^d}\leq \frac{\|\widetilde{u}_1\|_{H^1(D_0)^d}}{C}\sum_{j=2}^{\infty}\left(\frac{C}{\eta}\right)^{j}.
\]
The last expansion converges when $\eta>C$. Using (\ref{eq:4.35}) and (\ref{eq:4.36}) we conclude that there is a constant $C_1$ such that
\[
\left\Vert \sum_{j=0}^{\infty}\eta^{-j}u_j\right\Vert_{H^1(D)^d}\preceq C_1\left(\|f\|_{H^{-1}(D)^d}+\|g\|_{H^{1/2}(\partial D)^d}\right)\sum_{j=0}^{\infty}\left(\frac{C}{\eta}\right)^j.
\]
Moreover, the asymptotic limit $u_0$ satisfies (\ref{eq:18new}).
\end{proof}
Combining Lemmas \ref{mainlemma} to \ref{lemmaconver2} we get convergence for the expansion (\ref{eq:13}) with the boundary condition (\ref{eq:14}).

\begin{cor}
There are positive constants $C$ and $C_1$ such that for every $\eta>C$, we have
\[
\left\Vert u-\sum_{j=0}^{J}\eta^{-j}u_j\right\Vert_{H^1(D)^d}\leq C_1\left(\|f\|_H^{-1}(D)^d+\|g\|_{H^{1/2}(D)^d}\right)\sum_{j=J+1}^{\infty}\left(\frac{C}{\eta}\right)^j,
\]
for $J\geq 0$.
\end{cor}

\begin{rmk}
The case of several inclusions can be analyzed in similar way and it is not presented here, a description on how to perform this analysis for a scalar problem are given in \citet{calo2014asymptotic}.
\end{rmk}

\section{Conclusions}
\label{sec4:Conclusions}

We use asymptotic expansions to study high-contrast linear elasticity problems. In particular, we explain the procedure to compute the terms of the asymptotic expansion for $u_\eta$ with one stiff inclusion in linear elastic medium. We detail the analysis of the asymptotic power series for one highly inelastic inclusion. 





%
%
%
\bibliographystyle{elsarticle-harv} 
\bibliography{PANBibAELE}

\begin{thebibliography}{16}
\expandafter\ifx\csname natexlab\endcsname\relax\def\natexlab#1{#1}\fi
\expandafter\ifx\csname url\endcsname\relax
  \def\url#1{\texttt{#1}}\fi
\expandafter\ifx\csname urlprefix\endcsname\relax\def\urlprefix{URL }\fi

\bibitem[{Adams and Fournier(2003)}]{Adams-book2003}
Adams, R., Fournier, J., 2003. Sobolev spaces, 2nd Edition. Vol. 140 of Pure
  and Applied Mathematics. Elsevier/Academic Press, Amsterdam.

\bibitem[{Brezis(2010)}]{Brezis-book2010}
Brezis, H., 2010. Functional analysis, Sobolev spaces and partial differential
  equations. Springer.

\bibitem[{Calo et~al.(2014)Calo, Efendiev, and Galvis}]{calo2014asymptotic}
Calo, V.~M., Efendiev, Y., Galvis, J., 2014. Asymptotic expansions for
  high-contrast elliptic equations. Math. Models Methods Appl. Sci 24,
  465--494.

\bibitem[{Ciarlet(1988)}]{MR936420}
Ciarlet, P.~G., 1988. Mathematical elasticity. {V}ol. {I}. Vol.~20 of Studies
  in Mathematics and its Applications. North-Holland Publishing Co., Amsterdam.

\bibitem[{Ciarlet(1997)}]{MR1477663}
Ciarlet, P.~G., 1997. Mathematical elasticity. {V}ol. {II}. Vol.~27 of Studies
  in Mathematics and its Applications. North-Holland Publishing Co., Amsterdam.

\bibitem[{Di~Pietro and Nicaise(2013)}]{di2013locking}
Di~Pietro, D.~A., Nicaise, S., 2013. A locking-free discontinuous {G}alerkin
  method for linear elasticity in locally nearly incompressible heterogeneous
  media. Applied Numerical Mathematics 63, 105--116.

\bibitem[{Efendiev and Hou(2009)}]{MR2477579}
Efendiev, Y., Hou, T.~Y., 2009. Multiscale finite element methods. Vol.~4 of
  Surveys and Tutorials in the Applied Mathematical Sciences. Springer, New
  York.

\bibitem[{Gatica et~al.(2009)Gatica, M{\'a}rquez, and
  Meddahi}]{gatica2009augmented}
Gatica, G.~N., M{\'a}rquez, A., Meddahi, S., 2009. An augmented mixed finite
  element method for $3{D}$ linear elasticity problems. Journal of
  Computational and Applied Mathematics 231~(2), 526--540.

\bibitem[{Gonzalez and Stuart(2008)}]{gonzalez2008first}
Gonzalez, O., Stuart, A.~M., 2008. A first course in continuum mechanics.
  Cambridge University Press.

\bibitem[{Kang and Zhong-Ci(1996)}]{kang1996mathematical}
Kang, F., Zhong-Ci, S., 1996. Mathematical theory of elastic structures.
  Springer.

\bibitem[{Love(1944)}]{MR0010851}
Love, A. E.~H., 1944. A treatise on the mathematical theory of elasticity, 4th
  Edition. Dover Publications, New York.

\bibitem[{Malvern(1969)}]{malvern1969introduction}
Malvern, L.~E., 1969. Introduction to the mechanics of a continuous medium.
  Prentice-Hall Inc., Englewood Cliffs, New York.

\bibitem[{Sokolnikoff(1956)}]{MR0075755}
Sokolnikoff, I.~S., 1956. Mathematical theory of elasticity, 2nd Edition.
  McGraw-Hill Book Company, Inc., New York-Toronto-London.

\bibitem[{Wihler(2004)}]{wihler2004locking}
Wihler, T.~P., 2004. Locking-free {DGFEM} for elasticity problems in polygons.
  IMA Journal of Numerical Analysis 24~(1), 45--75.

\bibitem[{Xia et~al.(2014)Xia, Zhan, and Wei}]{xia2014mib}
Xia, K., Zhan, M., Wei, G.-W., 2014. {MIB} {G}alerkin method for elliptic
  interface problems. Journal of Computational and Applied Mathematics.

\bibitem[{Yang and Liu(1997)}]{yang1997least}
Yang, S.-Y., Liu, J.-L., 1997. Least-{S}quares finite element methods for the
  elasticity problem. Journal of Computational and Applied Mathematics 87~(1),
  39--60.

\end{thebibliography}




@Book{Solin-book2005,
  Title                    = {Partial differential equations and the finite element method},
  Author                   = {{\^S}ol{\'\i}n, Pavel},
  Publisher                = {Wiley},
  Year                     = {2005},

  Address                  = {New Jersey},
  Volume                   = {73},

  Owner                    = {Leonardo Andres},
  Timestamp                = {2013.09.15}
}

@Article{abdulle2006analysis,
  Title                    = {Analysis of a heterogeneous multiscale FEM for problems in elasticity},
  Author                   = {Abdulle, Assyr},
  Journal                  = {Mathematical Models and Methods in Applied Sciences},
  Year                     = {2006},
  Number                   = {04},
  Pages                    = {615--635},
  Volume                   = {16},

  Publisher                = {World Scientific}
}

@Book{Adams-book2003,
  Title                    = {Sobolev spaces},
  Author                   = {Adams, Robert and Fournier, John},
  Publisher                = {Elsevier/Academic Press},
  Year                     = {2003},

  Address                  = {Amsterdam},
  Edition                  = {2nd},
  Series                   = {Pure and Applied Mathematics},
  Volume                   = {140},

  Owner                    = {Leonardo Andres},
  Timestamp                = {2013.09.09}
}

@Article{akinola1999energy,
  Title                    = {An energy function for transversely-isotropic elastic material and the Ponyting Effect},
  Author                   = {Akinola, Ade},
  Journal                  = {Korean Journal of Computational \& Applied Mathematics},
  Year                     = {1999},
  Number                   = {3},
  Pages                    = {639--649},
  Volume                   = {6},

  Publisher                = {Springer}
}

@Book{atkinson2009theoretical,
  Title                    = {Theoretical numerical analysis: a functional analysis framework},
  Author                   = {Atkinson, Kendall},
  Publisher                = {Springer},
  Year                     = {2009},
  Volume                   = {39},

  Owner                    = {Leonardo Andres},
  Timestamp                = {2013.11.25}
}

@Article{babuvska1971error,
  Title                    = {Error-bounds for finite element method},
  Author                   = {Babu{\v{s}}ka, Ivo},
  Journal                  = {Numerische Mathematik},
  Year                     = {1971},
  Number                   = {4},
  Pages                    = {322--333},
  Volume                   = {16},

  Publisher                = {Springer}
}

@Book{MR2373954,
  Title                    = {The mathematical theory of finite element methods},
  Author                   = {Brenner, Susanne C. and Scott, L. Ridgway},
  Publisher                = {Springer, New York},
  Year                     = {2008},
  Edition                  = {Third},
  Series                   = {Texts in Applied Mathematics},
  Volume                   = {15},

  Doi                      = {10.1007/978-0-387-75934-0},
  ISBN                     = {978-0-387-75933-3},
  Pages                    = {xviii+397},
  Url                      = {http://dx.doi.org/10.1007/978-0-387-75934-0}
}

@Book{Brezis-book2010,
  Title                    = {Functional analysis, Sobolev spaces and partial differential equations},
  Author                   = {Brezis, Haim},
  Publisher                = {Springer},
  Year                     = {2010},

  Owner                    = {Leonardo Andres},
  Timestamp                = {2013.09.16}
}

@Article{brezzi2000discontinuous,
  Title                    = {Discontinuous Galerkin approximations for elliptic problems},
  Author                   = {Brezzi, Franco and Manzini, Gianmarco and Marini, Donatella and Pietra, Paola and Russo, Alessandro},
  Journal                  = {Numerical Methods for Partial Differential Equations},
  Year                     = {2000},
  Number                   = {4},
  Pages                    = {365--378},
  Volume                   = {16},

  Publisher                = {New York: Wiley, c1985-}
}

@InCollection{MR1777711,
  Title                    = {Extension theory for {S}obolev spaces on open sets with {L}ipschitz boundaries},
  Author                   = {Burenkov, Viktor I.},
  Booktitle                = {Nonlinear analysis, function spaces and applications, {V}ol. 6 ({P}rague, 1998)},
  Publisher                = {Acad. Sci. Czech Repub., Prague},
  Year                     = {1999},
  Pages                    = {1--49},

  Mrreviewer               = {Pavel A. Shvartsman}
}

@Article{calo2014asymptotic,
  Title                    = {Asymptotic expansions for high-contrast elliptic equations},
  Author                   = {Calo, Victor M and Efendiev, Yalchin and Galvis, Juan},
  Journal                  = {Math. Models Methods Appl. Sci},
  Year                     = {2014},
  Pages                    = {465--494},
  Volume                   = {24},

  Owner                    = {Leonardo Andrés},
  Timestamp                = {2013.11.30}
}

@Book{ciarlet1978finite,
  Title                    = {The finite element method for elliptic problems},
  Author                   = {Ciarlet, Philippe G.},
  Publisher                = {Classics in Applied Mathematics, Society for Industrial and Applied Mathematics (SIAM)},
  Year                     = {2002},

  Address                  = {Philadelphia, PA},
  Note                     = {Reprint of the 1978 original},
  Number                   = {1},
  Volume                   = {40},

  Journal                  = {Classics in Applied Mathematics, Society for Industrial and Applied Mathematics (SIAM)}
}

@Book{MR1477663,
  Title                    = {Mathematical elasticity. {V}ol. {II}},
  Author                   = {Ciarlet, Philippe G.},
  Publisher                = {North-Holland Publishing Co., Amsterdam},
  Year                     = {1997},
  Series                   = {Studies in Mathematics and its Applications},
  Volume                   = {27},

  ISBN                     = {0-444-82570-3},
  Mrreviewer               = {Michael S. Vogelius},
  Pages                    = {lxiv+497}
}

@Book{MR936420,
  Title                    = {Mathematical elasticity. {V}ol. {I}},
  Author                   = {Ciarlet, Philippe G.},
  Publisher                = {North-Holland Publishing Co., Amsterdam},
  Year                     = {1988},
  Series                   = {Studies in Mathematics and its Applications},
  Volume                   = {20},

  ISBN                     = {0-444-70259-8},
  Mrreviewer               = {Cornelius O. Horgan},
  Pages                    = {xlii+451}
}

@Article{di2013locking,
  Title                    = {A locking-free discontinuous {G}alerkin method for linear elasticity in locally nearly incompressible heterogeneous media},
  Author                   = {Di Pietro, Daniele A and Nicaise, Serge},
  Journal                  = {Applied Numerical Mathematics},
  Year                     = {2013},
  Pages                    = {105--116},
  Volume                   = {63},

  Publisher                = {Elsevier}
}

@InCollection{efendiev2012coarse,
  Title                    = {Coarse-grid multiscale model reduction techniques for flows in heterogeneous media and applications},
  Author                   = {Efendiev, Yalchin and Galvis, Juan},
  Booktitle                = {Numerical Analysis of Multiscale Problems},
  Publisher                = {Springer},
  Year                     = {2012},
  Pages                    = {97--125}
}

@Article{Efendiev2014generalized,
  Title                    = {Generalized Multiscale Finite Element Methods (GMsFEM)},
  Author                   = {Efendiev, Yalchin and Galvis, Juan and Hou T},
  Journal                  = {Journal of Computational Physics},
  Year                     = {2013},
  Pages                    = {116--135},
  Volume                   = {251},

  Owner                    = {Usuario},
  Timestamp                = {2014.02.06}
}

@Book{MR2477579,
  Title                    = {Multiscale finite element methods},
  Author                   = {Efendiev, Yalchin and Hou, Thomas Y.},
  Publisher                = {Springer, New York},
  Year                     = {2009},
  Series                   = {Surveys and Tutorials in the Applied Mathematical Sciences},
  Volume                   = {4},

  ISBN                     = {978-0-387-09495-3},
  Mrreviewer               = {Stefan Henn},
  Pages                    = {xii+234}
}

@Article{el2013fe,
  Title                    = {$FE^{2}$-multiscale in linear elasticity based on parametrized microscale models using proper generalized decomposition},
  Author                   = {El Halabi, F and Gonz{\'a}lez, D and Chico, A and Doblar{\'e}, M},
  Journal                  = {Computer Methods in Applied Mechanics and Engineering},
  Year                     = {2013},
  Pages                    = {183--202},
  Volume                   = {257},

  Publisher                = {Elsevier}
}

@Book{Evans-book1990,
  Title                    = {Partial differential equations},
  Author                   = {Evans, Lawrence},
  Publisher                = {American Mathematical Society},
  Year                     = {1990},

  Address                  = {Rhode Island},
  Series                   = {Graduate studies in Mathematics},
  Volume                   = {19},

  Owner                    = {Leonardo Andres},
  Timestamp                = {2013.09.09}
}

@Book{Galvis-book2009,
  Title                    = {Introdu\c{c}\~{a}o aos m\'{e}todos de decomposi\c{c}\~{a}o de dom\'{i}nio},
  Author                   = {Galvis, Juan},
  Publisher                = {Lecture notes for a minicourse in the 27 Col\'oquio Brasileiro de Matem\'atica IMPA},
  Year                     = {2009},

  Owner                    = {Leonardo Andres},
  Timestamp                = {2014.01.11}
}

@Article{galvis2014spec,
  Title                    = {Spectral multiscale finite element for nonlinear flows in highly heterogeneous media: A reduced basis approach},
  Author                   = {Galvis, Juan and Ki Kang, S},
  Journal                  = {Journal of Computational and Applied Mathematics},
  Year                     = {2014},
  Pages                    = {494--508},
  Volume                   = {260},

  Publisher                = {Elsevier}
}

@Unpublished{Galvis-poveda2014,
  Title                    = {Using Weak Forms to Derive Asymptotic Expansions of Elliptic Equations with High Contrast Coefficients},
  Author                   = {Galvis, Juan and Poveda, Leonardo},
  Note                     = {Submitted},

  Month                    = {February},
  Year                     = {2014},

  Owner                    = {Leonardo Andrés},
  Timestamp                = {2013.12.02}
}

@Conference{galvis-povedaSimposio,
  Title                    = {Expansiones Asint\'oticas de Ecuaciones El\'ipticas en medios de Alto Contraste},
  Author                   = {Galvis, Juan and Poveda, Leonardo},
  Booktitle                = {Memorias VIII Simposio Nororiental de Matem\'aticas},
  Year                     = {2013},

  Address                  = {Universidad Industrial de Santander},

  Owner                    = {Leonardo Andrés},
  Timestamp                = {2014.03.06},
  Url                      = {http://matematicas.uis.edu.co/8simposio/sites/default/files/Memorias8SNMb.pdf}
}

@Book{Galvis-book2011,
  Title                    = {Introdu\c{c}\~{a}o \`{a} aproxima\c{c}\~{a}o num\'erica de equa\c{c}\~{o}es diferenciais parciais via m\'{e}todo elementos finitos},
  Author                   = {Galvis, Juan and Versieux, Henrique},
  Publisher                = {Lecture notes for a minicourse in the 28 Col\'oquio Brasileiro de Matem\'atica IMPA},
  Year                     = {2011},

  Owner                    = {Leonardo Andres},
  Timestamp                = {2014.01.11}
}

@Article{gatica2009augmented,
  Title                    = {An augmented mixed finite element method for $3{D}$ linear elasticity problems},
  Author                   = {Gatica, Gabriel N and M{\'a}rquez, Antonio and Meddahi, Salim},
  Journal                  = {Journal of Computational and Applied Mathematics},
  Year                     = {2009},
  Number                   = {2},
  Pages                    = {526--540},
  Volume                   = {231},

  Publisher                = {Elsevier}
}

@Book{gonzalez2008first,
  Title                    = {A first course in continuum mechanics},
  Author                   = {Gonzalez, Oscar and Stuart, Andrew M},
  Publisher                = {Cambridge University Press},
  Year                     = {2008}
}

@Book{Grisvard-book1985,
  Title                    = {Elliptic Problems in nonsmooth domains},
  Author                   = {Grisvard, Pierre},
  Publisher                = {Pitman (Advanced Publishing Program)},
  Year                     = {1985},

  Address                  = {Boston, MA},
  Volume                   = {24},

  Owner                    = {Leonardo Andres},
  Pages                    = {Monographs and Studies in Mathematics},
  Timestamp                = {2013.09.09}
}

@Misc{hecht2005freefem++,
  Title                    = {FreeFem++ manual},

  Author                   = {Hecht, Fr{\'e}d{\'e}ric and Pironneau, Olivier and Le Hyaric, A and Ohtsuka, K},
  Year                     = {2005}
}

@Book{johnson2012numerical,
  Title                    = {Numerical solution of partial differential equations by the finite element method},
  Author                   = {Johnson, Claes},
  Publisher                = {Dover Publications, Inc., Mineola, NY},
  Year                     = {2009},
  Note                     = {Reprint of the 1987 edition},

  ISBN                     = {978-0-486-46900-3; 0-486-46900-X},
  Pages                    = {ii+279}
}

@Book{kang1996mathematical,
  Title                    = {Mathematical theory of elastic structures},
  Author                   = {Kang, Feng and Zhong-Ci, Shi},
  Publisher                = {Springer},
  Year                     = {1996}
}

@Book{Kesavan-book1989,
  Title                    = {Topics in functional analysis and applications},
  Author                   = {Kesavan, Srinivasan},
  Publisher                = {Wiley},
  Year                     = {1989},

  Address                  = {New York},

  Owner                    = {Leonardo Andres},
  Timestamp                = {2013.09.15}
}

@Book{lions1972non,
  Title                    = {Non-homogeneous boundary value problems and applications},
  Author                   = {Lions, Jacques Louis and Magenes, Enrico and Kenneth, P.},
  Publisher                = {Springer Berlin},
  Year                     = {1972},
  Volume                   = {1}
}

@Book{lions1970problemes,
  Title                    = {Problemes aux limites non homogenes et applications},
  Author                   = {Lions, J-L. and Magenes, E.},
  Publisher                = {Dunod Paris},
  Year                     = {1970},
  Volume                   = {1}
}

@Book{MR0010851,
  Title                    = {A treatise on the mathematical theory of elasticity},
  Author                   = {Love, A. E. H.},
  Publisher                = {Dover Publications, New York},
  Year                     = {1944},
  Edition                  = {4},

  Pages                    = {xviii+643}
}

@Book{malvern1969introduction,
  Title                    = {Introduction to the mechanics of a continuous medium},
  Author                   = {Malvern, Lawrence E},
  Publisher                = {Prentice-Hall Inc.},
  Year                     = {1969},

  Address                  = {Englewood Cliffs, New York}
}

@Book{mclean2000strongly,
  Title                    = {Strongly elliptic systems and boundary integral equations},
  Author                   = {McLean, William},
  Publisher                = {Cambridge university press},
  Year                     = {2000}
}

@Book{Neecas-book2012,
  Title                    = {Direct methods in the theory of elliptic equations},
  Author                   = {Ne{\v{c}}as, Jind{\v{r}}ich},
  Publisher                = {Springer, Heidelberg},
  Year                     = {2012},
  Note                     = {Translated from the 1967 French original by Gerard Tronel and Alois Kufner, Editorial coordination and preface by {\v{S}}{\'a}rka Ne{\v{c}}asov{\'a} and a contribution by Christian G. Simader},
  Series                   = {Springer Monographs in Mathematics},

  Doi                      = {10.1007/978-3-642-10455-8},
  ISBN                     = {978-3-642-10454-1; 978-3-642-10455-8},
  Mrreviewer               = {Norbert Ortner},
  Pages                    = {xvi+372},
  Url                      = {http://dx.doi.org/10.1007/978-3-642-10455-8}
}

@Unpublished{Poveda2014MsFEMapl,
  Title                    = {Elliptic Equations with High-Contrast Coefficients and Applications in Multiscale Finite Element Methods},
  Author                   = {Poveda, Leonardo and Galvis, Juan. and Calo, Victor M. and Huepo, Sebastian},
  Note                     = {Submitted},

  Month                    = {September},
  Year                     = {2014},

  Owner                    = {Leonardo Andrés},
  Timestamp                = {2014.09.16},
  Url                      = {http://arxiv.org/abs/1410.0293}
}

@Article{santos2011finite,
  Title                    = {Finite element approximation of coupled seismic and electromagnetic waves in fluid-saturated poroviscoelastic media},
  Author                   = {Santos, Juan E},
  Journal                  = {Numerical Methods for Partial Differential Equations},
  Year                     = {2011},
  Number                   = {2},
  Pages                    = {351--386},
  Volume                   = {27},

  Publisher                = {Wiley Online Library}
}

@Book{slaughter2002linearized,
  Title                    = {The linearized theory of elasticity},
  Author                   = {Slaughter, William S},
  Publisher                = {Springer},
  Year                     = {2002}
}

@Book{MR0075755,
  Title                    = {Mathematical theory of elasticity},
  Author                   = {Sokolnikoff, Ivan S},
  Publisher                = {McGraw-Hill Book Company, Inc., New York-Toronto-London},
  Year                     = {1956},
  Edition                  = {2},

  Mrreviewer               = {J. R. M. Radok},
  Pages                    = {xi+476}
}

@Book{Tartar-book2000,
  Title                    = {An Introduction to Sobolev Spaces and Interpolation Spaces},
  Author                   = {Tartar, Luc},
  Publisher                = {Springer},
  Year                     = {2000},

  Address                  = {Pittsburgh},

  Owner                    = {Leonardo Andres},
  Timestamp                = {2013.09.10}
}

@Book{toselli2005domain,
  Title                    = {Domain decomposition methods: algorithms and theory},
  Author                   = {Toselli, Andrea and Widlund, Olof},
  Publisher                = {Springer},
  Year                     = {2005},
  Volume                   = {34}
}

@Article{wihler2004locking,
  Title                    = {Locking-free {DGFEM} for elasticity problems in polygons},
  Author                   = {Wihler, Thomas P},
  Journal                  = {IMA Journal of Numerical Analysis},
  Year                     = {2004},
  Number                   = {1},
  Pages                    = {45--75},
  Volume                   = {24},

  Publisher                = {Oxford University Press}
}

@Article{xia2014mib,
  Title                    = {{MIB} {G}alerkin method for elliptic interface problems},
  Author                   = {Xia, Kelin and Zhan, Meng and Wei, Guo-Wei},
  Journal                  = {Journal of Computational and Applied Mathematics},
  Year                     = {2014},

  Publisher                = {Elsevier}
}

@Article{xu2007delta,
  Title                    = {$\delta$-mapping algorithm coupled with WENO reconstruction for nonlinear elasticity in heterogeneous media},
  Author                   = {Xu, Zhenli and Zhang, Peng and Liu, Ruxun},
  Journal                  = {Applied numerical mathematics},
  Year                     = {2007},
  Number                   = {1},
  Pages                    = {103--116},
  Volume                   = {57},

  Publisher                = {Elsevier}
}

@Article{yang1997least,
  Title                    = {Least-{S}quares finite element methods for the elasticity problem},
  Author                   = {Yang, Suh-Yuh and Liu, Jinn-Liang},
  Journal                  = {Journal of Computational and Applied Mathematics},
  Year                     = {1997},
  Number                   = {1},
  Pages                    = {39--60},
  Volume                   = {87},

  Publisher                = {Elsevier}
}

@Article{yi2014convergence,
  Title                    = {Convergence analysis of a new mixed finite element method for Biot's consolidation model},
  Author                   = {Yi, Son-Young},
  Journal                  = {Numerical Methods for Partial Differential Equations},
  Year                     = {2014},
  Number                   = {4},
  Pages                    = {1189--1210},
  Volume                   = {30},

  Publisher                = {Wiley Online Library}
}

@Book{ZeidlerFunca,
  Title                    = {Applied functional analysis},
  Author                   = {Zeidler, E.},
  Volume                   = {108},

  Journal                  = {Applied Mathematical Sciences}
}





@Book{Solin-book2005,
  Title                    = {Partial differential equations and the finite element method},
  Author                   = {{\^S}ol{\'\i}n, Pavel},
  Publisher                = {Wiley},
  Year                     = {2005},

  Address                  = {New Jersey},
  Volume                   = {73},

  Owner                    = {Leonardo Andres},
  Timestamp                = {2013.09.15}
}

@Article{abdulle2006analysis,
  Title                    = {Analysis of a heterogeneous multiscale FEM for problems in elasticity},
  Author                   = {Abdulle, Assyr},
  Journal                  = {Mathematical Models and Methods in Applied Sciences},
  Year                     = {2006},
  Number                   = {04},
  Pages                    = {615--635},
  Volume                   = {16},

  Publisher                = {World Scientific}
}

@Book{Adams-book2003,
  Title                    = {Sobolev spaces},
  Author                   = {Adams, Robert and Fournier, John},
  Publisher                = {Elsevier/Academic Press},
  Year                     = {2003},

  Address                  = {Amsterdam},
  Edition                  = {2nd},
  Series                   = {Pure and Applied Mathematics},
  Volume                   = {140},

  Owner                    = {Leonardo Andres},
  Timestamp                = {2013.09.09}
}

@Article{akinola1999energy,
  Title                    = {An energy function for transversely-isotropic elastic material and the Ponyting Effect},
  Author                   = {Akinola, Ade},
  Journal                  = {Korean Journal of Computational \& Applied Mathematics},
  Year                     = {1999},
  Number                   = {3},
  Pages                    = {639--649},
  Volume                   = {6},

  Publisher                = {Springer}
}

@Book{atkinson2009theoretical,
  Title                    = {Theoretical numerical analysis: a functional analysis framework},
  Author                   = {Atkinson, Kendall},
  Publisher                = {Springer},
  Year                     = {2009},
  Volume                   = {39},

  Owner                    = {Leonardo Andres},
  Timestamp                = {2013.11.25}
}

@Article{babuvska1971error,
  Title                    = {Error-bounds for finite element method},
  Author                   = {Babu{\v{s}}ka, Ivo},
  Journal                  = {Numerische Mathematik},
  Year                     = {1971},
  Number                   = {4},
  Pages                    = {322--333},
  Volume                   = {16},

  Publisher                = {Springer}
}

@Book{MR2373954,
  Title                    = {The mathematical theory of finite element methods},
  Author                   = {Brenner, Susanne C. and Scott, L. Ridgway},
  Publisher                = {Springer, New York},
  Year                     = {2008},
  Edition                  = {Third},
  Series                   = {Texts in Applied Mathematics},
  Volume                   = {15},

  Doi                      = {10.1007/978-0-387-75934-0},
  ISBN                     = {978-0-387-75933-3},
  Pages                    = {xviii+397},
  Url                      = {http://dx.doi.org/10.1007/978-0-387-75934-0}
}

@Book{Brezis-book2010,
  Title                    = {Functional analysis, Sobolev spaces and partial differential equations},
  Author                   = {Brezis, Haim},
  Publisher                = {Springer},
  Year                     = {2010},

  Owner                    = {Leonardo Andres},
  Timestamp                = {2013.09.16}
}

@Article{brezzi2000discontinuous,
  Title                    = {Discontinuous Galerkin approximations for elliptic problems},
  Author                   = {Brezzi, Franco and Manzini, Gianmarco and Marini, Donatella and Pietra, Paola and Russo, Alessandro},
  Journal                  = {Numerical Methods for Partial Differential Equations},
  Year                     = {2000},
  Number                   = {4},
  Pages                    = {365--378},
  Volume                   = {16},

  Publisher                = {New York: Wiley, c1985-}
}

@InCollection{MR1777711,
  Title                    = {Extension theory for {S}obolev spaces on open sets with {L}ipschitz boundaries},
  Author                   = {Burenkov, Viktor I.},
  Booktitle                = {Nonlinear analysis, function spaces and applications, {V}ol. 6 ({P}rague, 1998)},
  Publisher                = {Acad. Sci. Czech Repub., Prague},
  Year                     = {1999},
  Pages                    = {1--49},

  Mrreviewer               = {Pavel A. Shvartsman}
}

@Article{calo2014asymptotic,
  Title                    = {Asymptotic expansions for high-contrast elliptic equations},
  Author                   = {Calo, Victor M and Efendiev, Yalchin and Galvis, Juan},
  Journal                  = {Math. Models Methods Appl. Sci},
  Year                     = {2014},
  Pages                    = {465--494},
  Volume                   = {24},

  Owner                    = {Leonardo Andrés},
  Timestamp                = {2013.11.30}
}

@Book{ciarlet1978finite,
  Title                    = {The finite element method for elliptic problems},
  Author                   = {Ciarlet, Philippe G.},
  Publisher                = {Classics in Applied Mathematics, Society for Industrial and Applied Mathematics (SIAM)},
  Year                     = {2002},

  Address                  = {Philadelphia, PA},
  Note                     = {Reprint of the 1978 original},
  Number                   = {1},
  Volume                   = {40},

  Journal                  = {Classics in Applied Mathematics, Society for Industrial and Applied Mathematics (SIAM)}
}

@Book{MR1477663,
  Title                    = {Mathematical elasticity. {V}ol. {II}},
  Author                   = {Ciarlet, Philippe G.},
  Publisher                = {North-Holland Publishing Co., Amsterdam},
  Year                     = {1997},
  Series                   = {Studies in Mathematics and its Applications},
  Volume                   = {27},

  ISBN                     = {0-444-82570-3},
  Mrreviewer               = {Michael S. Vogelius},
  Pages                    = {lxiv+497}
}

@Book{MR936420,
  Title                    = {Mathematical elasticity. {V}ol. {I}},
  Author                   = {Ciarlet, Philippe G.},
  Publisher                = {North-Holland Publishing Co., Amsterdam},
  Year                     = {1988},
  Series                   = {Studies in Mathematics and its Applications},
  Volume                   = {20},

  ISBN                     = {0-444-70259-8},
  Mrreviewer               = {Cornelius O. Horgan},
  Pages                    = {xlii+451}
}

@Article{di2013locking,
  Title                    = {A locking-free discontinuous {G}alerkin method for linear elasticity in locally nearly incompressible heterogeneous media},
  Author                   = {Di Pietro, Daniele A and Nicaise, Serge},
  Journal                  = {Applied Numerical Mathematics},
  Year                     = {2013},
  Pages                    = {105--116},
  Volume                   = {63},

  Publisher                = {Elsevier}
}

@InCollection{efendiev2012coarse,
  Title                    = {Coarse-grid multiscale model reduction techniques for flows in heterogeneous media and applications},
  Author                   = {Efendiev, Yalchin and Galvis, Juan},
  Booktitle                = {Numerical Analysis of Multiscale Problems},
  Publisher                = {Springer},
  Year                     = {2012},
  Pages                    = {97--125}
}

@Article{Efendiev2014generalized,
  Title                    = {Generalized Multiscale Finite Element Methods (GMsFEM)},
  Author                   = {Efendiev, Yalchin and Galvis, Juan and Hou T},
  Journal                  = {Journal of Computational Physics},
  Year                     = {2013},
  Pages                    = {116--135},
  Volume                   = {251},

  Owner                    = {Usuario},
  Timestamp                = {2014.02.06}
}

@Book{MR2477579,
  Title                    = {Multiscale finite element methods},
  Author                   = {Efendiev, Yalchin and Hou, Thomas Y.},
  Publisher                = {Springer, New York},
  Year                     = {2009},
  Series                   = {Surveys and Tutorials in the Applied Mathematical Sciences},
  Volume                   = {4},

  ISBN                     = {978-0-387-09495-3},
  Mrreviewer               = {Stefan Henn},
  Pages                    = {xii+234}
}

@Article{el2013fe,
  Title                    = {$FE^{2}$-multiscale in linear elasticity based on parametrized microscale models using proper generalized decomposition},
  Author                   = {El Halabi, F and Gonz{\'a}lez, D and Chico, A and Doblar{\'e}, M},
  Journal                  = {Computer Methods in Applied Mechanics and Engineering},
  Year                     = {2013},
  Pages                    = {183--202},
  Volume                   = {257},

  Publisher                = {Elsevier}
}

@Book{Evans-book1990,
  Title                    = {Partial differential equations},
  Author                   = {Evans, Lawrence},
  Publisher                = {American Mathematical Society},
  Year                     = {1990},

  Address                  = {Rhode Island},
  Series                   = {Graduate studies in Mathematics},
  Volume                   = {19},

  Owner                    = {Leonardo Andres},
  Timestamp                = {2013.09.09}
}

@Book{Galvis-book2009,
  Title                    = {Introdu\c{c}\~{a}o aos m\'{e}todos de decomposi\c{c}\~{a}o de dom\'{i}nio},
  Author                   = {Galvis, Juan},
  Publisher                = {Lecture notes for a minicourse in the 27 Col\'oquio Brasileiro de Matem\'atica IMPA},
  Year                     = {2009},

  Owner                    = {Leonardo Andres},
  Timestamp                = {2014.01.11}
}

@Article{galvis2014spec,
  Title                    = {Spectral multiscale finite element for nonlinear flows in highly heterogeneous media: A reduced basis approach},
  Author                   = {Galvis, Juan and Ki Kang, S},
  Journal                  = {Journal of Computational and Applied Mathematics},
  Year                     = {2014},
  Pages                    = {494--508},
  Volume                   = {260},

  Publisher                = {Elsevier}
}

@Unpublished{Galvis-poveda2014,
  Title                    = {Using Weak Forms to Derive Asymptotic Expansions of Elliptic Equations with High Contrast Coefficients},
  Author                   = {Galvis, Juan and Poveda, Leonardo},
  Note                     = {Submitted},

  Month                    = {February},
  Year                     = {2014},

  Owner                    = {Leonardo Andrés},
  Timestamp                = {2013.12.02}
}

@Conference{galvis-povedaSimposio,
  Title                    = {Expansiones Asint\'oticas de Ecuaciones El\'ipticas en medios de Alto Contraste},
  Author                   = {Galvis, Juan and Poveda, Leonardo},
  Booktitle                = {Memorias VIII Simposio Nororiental de Matem\'aticas},
  Year                     = {2013},

  Address                  = {Universidad Industrial de Santander},

  Owner                    = {Leonardo Andrés},
  Timestamp                = {2014.03.06},
  Url                      = {http://matematicas.uis.edu.co/8simposio/sites/default/files/Memorias8SNMb.pdf}
}

@Book{Galvis-book2011,
  Title                    = {Introdu\c{c}\~{a}o \`{a} aproxima\c{c}\~{a}o num\'erica de equa\c{c}\~{o}es diferenciais parciais via m\'{e}todo elementos finitos},
  Author                   = {Galvis, Juan and Versieux, Henrique},
  Publisher                = {Lecture notes for a minicourse in the 28 Col\'oquio Brasileiro de Matem\'atica IMPA},
  Year                     = {2011},

  Owner                    = {Leonardo Andres},
  Timestamp                = {2014.01.11}
}

@Article{gatica2009augmented,
  Title                    = {An augmented mixed finite element method for $3{D}$ linear elasticity problems},
  Author                   = {Gatica, Gabriel N and M{\'a}rquez, Antonio and Meddahi, Salim},
  Journal                  = {Journal of Computational and Applied Mathematics},
  Year                     = {2009},
  Number                   = {2},
  Pages                    = {526--540},
  Volume                   = {231},

  Publisher                = {Elsevier}
}

@Book{gonzalez2008first,
  Title                    = {A first course in continuum mechanics},
  Author                   = {Gonzalez, Oscar and Stuart, Andrew M},
  Publisher                = {Cambridge University Press},
  Year                     = {2008}
}

@Book{Grisvard-book1985,
  Title                    = {Elliptic Problems in nonsmooth domains},
  Author                   = {Grisvard, Pierre},
  Publisher                = {Pitman (Advanced Publishing Program)},
  Year                     = {1985},

  Address                  = {Boston, MA},
  Volume                   = {24},

  Owner                    = {Leonardo Andres},
  Pages                    = {Monographs and Studies in Mathematics},
  Timestamp                = {2013.09.09}
}

@Misc{hecht2005freefem++,
  Title                    = {FreeFem++ manual},

  Author                   = {Hecht, Fr{\'e}d{\'e}ric and Pironneau, Olivier and Le Hyaric, A and Ohtsuka, K},
  Year                     = {2005}
}

@Book{johnson2012numerical,
  Title                    = {Numerical solution of partial differential equations by the finite element method},
  Author                   = {Johnson, Claes},
  Publisher                = {Dover Publications, Inc., Mineola, NY},
  Year                     = {2009},
  Note                     = {Reprint of the 1987 edition},

  ISBN                     = {978-0-486-46900-3; 0-486-46900-X},
  Pages                    = {ii+279}
}

@Book{kang1996mathematical,
  Title                    = {Mathematical theory of elastic structures},
  Author                   = {Kang, Feng and Zhong-Ci, Shi},
  Publisher                = {Springer},
  Year                     = {1996}
}

@Book{Kesavan-book1989,
  Title                    = {Topics in functional analysis and applications},
  Author                   = {Kesavan, Srinivasan},
  Publisher                = {Wiley},
  Year                     = {1989},

  Address                  = {New York},

  Owner                    = {Leonardo Andres},
  Timestamp                = {2013.09.15}
}

@Book{lions1972non,
  Title                    = {Non-homogeneous boundary value problems and applications},
  Author                   = {Lions, Jacques Louis and Magenes, Enrico and Kenneth, P.},
  Publisher                = {Springer Berlin},
  Year                     = {1972},
  Volume                   = {1}
}

@Book{lions1970problemes,
  Title                    = {Problemes aux limites non homogenes et applications},
  Author                   = {Lions, J-L. and Magenes, E.},
  Publisher                = {Dunod Paris},
  Year                     = {1970},
  Volume                   = {1}
}

@Book{MR0010851,
  Title                    = {A treatise on the mathematical theory of elasticity},
  Author                   = {Love, A. E. H.},
  Publisher                = {Dover Publications, New York},
  Year                     = {1944},
  Edition                  = {4},

  Pages                    = {xviii+643}
}

@Book{malvern1969introduction,
  Title                    = {Introduction to the mechanics of a continuous medium},
  Author                   = {Malvern, Lawrence E},
  Publisher                = {Prentice-Hall Inc.},
  Year                     = {1969},

  Address                  = {Englewood Cliffs, New York}
}

@Book{mclean2000strongly,
  Title                    = {Strongly elliptic systems and boundary integral equations},
  Author                   = {McLean, William},
  Publisher                = {Cambridge university press},
  Year                     = {2000}
}

@Book{Neecas-book2012,
  Title                    = {Direct methods in the theory of elliptic equations},
  Author                   = {Ne{\v{c}}as, Jind{\v{r}}ich},
  Publisher                = {Springer, Heidelberg},
  Year                     = {2012},
  Note                     = {Translated from the 1967 French original by Gerard Tronel and Alois Kufner, Editorial coordination and preface by {\v{S}}{\'a}rka Ne{\v{c}}asov{\'a} and a contribution by Christian G. Simader},
  Series                   = {Springer Monographs in Mathematics},

  Doi                      = {10.1007/978-3-642-10455-8},
  ISBN                     = {978-3-642-10454-1; 978-3-642-10455-8},
  Mrreviewer               = {Norbert Ortner},
  Pages                    = {xvi+372},
  Url                      = {http://dx.doi.org/10.1007/978-3-642-10455-8}
}

@Unpublished{Poveda2014MsFEMapl,
  Title                    = {Elliptic Equations with High-Contrast Coefficients and Applications in Multiscale Finite Element Methods},
  Author                   = {Poveda, Leonardo and Galvis, Juan. and Calo, Victor M. and Huepo, Sebastian},
  Note                     = {Submitted},

  Month                    = {September},
  Year                     = {2014},

  Owner                    = {Leonardo Andrés},
  Timestamp                = {2014.09.16},
  Url                      = {http://arxiv.org/abs/1410.0293}
}

@Article{santos2011finite,
  Title                    = {Finite element approximation of coupled seismic and electromagnetic waves in fluid-saturated poroviscoelastic media},
  Author                   = {Santos, Juan E},
  Journal                  = {Numerical Methods for Partial Differential Equations},
  Year                     = {2011},
  Number                   = {2},
  Pages                    = {351--386},
  Volume                   = {27},

  Publisher                = {Wiley Online Library}
}

@Book{slaughter2002linearized,
  Title                    = {The linearized theory of elasticity},
  Author                   = {Slaughter, William S},
  Publisher                = {Springer},
  Year                     = {2002}
}

@Book{MR0075755,
  Title                    = {Mathematical theory of elasticity},
  Author                   = {Sokolnikoff, Ivan S},
  Publisher                = {McGraw-Hill Book Company, Inc., New York-Toronto-London},
  Year                     = {1956},
  Edition                  = {2},

  Mrreviewer               = {J. R. M. Radok},
  Pages                    = {xi+476}
}

@Book{Tartar-book2000,
  Title                    = {An Introduction to Sobolev Spaces and Interpolation Spaces},
  Author                   = {Tartar, Luc},
  Publisher                = {Springer},
  Year                     = {2000},

  Address                  = {Pittsburgh},

  Owner                    = {Leonardo Andres},
  Timestamp                = {2013.09.10}
}

@Book{toselli2005domain,
  Title                    = {Domain decomposition methods: algorithms and theory},
  Author                   = {Toselli, Andrea and Widlund, Olof},
  Publisher                = {Springer},
  Year                     = {2005},
  Volume                   = {34}
}

@Article{wihler2004locking,
  Title                    = {Locking-free {DGFEM} for elasticity problems in polygons},
  Author                   = {Wihler, Thomas P},
  Journal                  = {IMA Journal of Numerical Analysis},
  Year                     = {2004},
  Number                   = {1},
  Pages                    = {45--75},
  Volume                   = {24},

  Publisher                = {Oxford University Press}
}

@Article{xia2014mib,
  Title                    = {{MIB} {G}alerkin method for elliptic interface problems},
  Author                   = {Xia, Kelin and Zhan, Meng and Wei, Guo-Wei},
  Journal                  = {Journal of Computational and Applied Mathematics},
  Year                     = {2014},

  Publisher                = {Elsevier}
}

@Article{xu2007delta,
  Title                    = {$\delta$-mapping algorithm coupled with WENO reconstruction for nonlinear elasticity in heterogeneous media},
  Author                   = {Xu, Zhenli and Zhang, Peng and Liu, Ruxun},
  Journal                  = {Applied numerical mathematics},
  Year                     = {2007},
  Number                   = {1},
  Pages                    = {103--116},
  Volume                   = {57},

  Publisher                = {Elsevier}
}

@Article{yang1997least,
  Title                    = {Least-{S}quares finite element methods for the elasticity problem},
  Author                   = {Yang, Suh-Yuh and Liu, Jinn-Liang},
  Journal                  = {Journal of Computational and Applied Mathematics},
  Year                     = {1997},
  Number                   = {1},
  Pages                    = {39--60},
  Volume                   = {87},

  Publisher                = {Elsevier}
}

@Article{yi2014convergence,
  Title                    = {Convergence analysis of a new mixed finite element method for Biot's consolidation model},
  Author                   = {Yi, Son-Young},
  Journal                  = {Numerical Methods for Partial Differential Equations},
  Year                     = {2014},
  Number                   = {4},
  Pages                    = {1189--1210},
  Volume                   = {30},

  Publisher                = {Wiley Online Library}
}

@Book{ZeidlerFunca,
  Title                    = {Applied functional analysis},
  Author                   = {Zeidler, E.},
  Volume                   = {108},

  Journal                  = {Applied Mathematical Sciences}
}

\end{document}